\documentclass[11pt]{article}
\usepackage{graphicx}
\usepackage{amsthm, amsmath, amssymb, tikz, bm}
\usepackage{mathtools}
\usepackage{xifthen}
\usepackage{caption}
\usepackage[colorlinks=true,
linkcolor=blue,citecolor=blue,
urlcolor=blue]{hyperref}

\usepackage[margin=0.8in]{geometry} 

\usepackage[shortlabels]{enumitem}
\usepackage{todonotes}
\usetikzlibrary{calc,shapes, backgrounds}
\allowdisplaybreaks

\newtheorem{lemma}{Lemma}
\newtheorem{theorem}{Theorem}

\newtheorem{claim}{Claim}

\newcommand{\dss}{\displaystyle\sum}

\newcommand{\lp}{\left(}
\newcommand{\rp}{\right)}

\newcommand{\cx}{{\bf x}}
\newcommand{\cw}{{\bf w}}
\newcommand{\cz}{{\bf z}}
\newcommand{\vol}{{\rm Vol}}

\DeclarePairedDelimiter{\floor}{\lfloor}{\rfloor}

\DeclarePairedDelimiter{\abs}{\lvert}{\rvert}%

\newcommand{\WL}[1]{\textcolor{magenta}{{}Comment by WL: #1}}
\newcommand{\ZW}[1]{\textcolor{purple}{{}Comment by ZW: #1}}

\newcommand{\bx}{\mathbf{x}}

\newcommand{\one}{\mathbf{1}}

\title{Maximum spread of $K_{2,t}$-minor-free graphs}
\author{
William Linz, \thanks{University of South Carolina, Columbia, SC. ({\tt wlinz@mailbox.sc.edu}). The author is partially supported by NSF DMS 2038080 grant.}
\and Linyuan Lu, \thanks{University of South Carolina, Columbia, SC. ({\tt lu@math.sc.edu}). The author is partially supported by NSF DMS 2038080 grant.}
\and Zhiyu Wang \thanks{Georgia Institute of Technology, Atlanta, GA.
({\tt zwang672@gatech.edu}).}
}

\begin{document}

\maketitle
\abstract{
The spread of a graph $G$ is the difference between the largest and smallest eigenvalues of the adjacency matrix of $G$. In this paper, we consider
the family of graphs which contain no $K_{2,t}$-minor.
We show that for any $t\geq 2$, there is an integer $\xi_t$ such that
the maximum spread of an $n$-vertex $K_{2,t}$-minor-free graph is achieved by the graph obtained by joining a vertex to the disjoint union of $\lfloor \frac{2n+\xi_t}{3t}\rfloor$ copies of $K_t$ and
$n-1 - t\lfloor \frac{2n+\xi_t}{3t}\rfloor$
isolated vertices. The extremal graph is unique, except when $t\equiv 4 \pmod {12}$ and
$\frac{2n+ \xi_t} {3t}$ is an integer, in which case the other extremal graph is the graph obtained by joining a vertex to the disjoint union of $\lfloor \frac{2n+\xi_t}{3t}\rfloor-1$ copies of $K_t$ and $n-1-t(\lfloor \frac{2n+\xi_t}{3t}\rfloor-1)$ isolated vertices.
Furthermore, we give an explicit formula for $\xi_t$.
}

\section{Introduction}
Given a square matrix $M$, the \textit{spread} of $M$, denoted by $S(M)$, is defined as $S(M):= \max_{i,j} |\lambda_i -\lambda_j|$, where the maximum is taken over all pairs of eigenvalues of $M$. In other words, $S(M)$ is the diameter of the spectrum of $M$.
Given a graph $G=(V,E)$ on $n$ vertices, the \textit{spread} of $G$, denoted by $S(G)$, is defined as the spread of the adjacency matrix $A(G)$ of $G$. Let $\lambda_1(G) \geq \cdots \geq \lambda_n(G)$ be the eigenvalues of $A(G)$. Here $\lambda_1$ is called the $\textit{spectral radius}$ of $G$. Since $A(G)$ is a real symmetric matrix, we have that the $\lambda_i$s are all real numbers. Thus $S(G) = \lambda_1 -\lambda_n$.

The systematic study of the spread of graphs was initiated by Gregory, Hershkowitz, and Kirkland \cite{GHK2001}. One of the central focuses of this area is to find the maximum or minimum spread over a fixed family of graphs and characterize the extremal graphs. Problems of such extremal flavor have been investigated for trees~\cite{AP2015}, graphs with few cycles~\cite{FWG2008, PBA2009, Wu-Shu2010}, the family of all $n$-vertex graphs~\cite{Aouchiche2008, BRTU2021+, Riasanovsky2021,Stanic2015, Stevanovic2014, Urschel2021}, the family of bipartite graphs~\cite{BRTU2021+}, graphs with a given matching number~\cite{LZZ2007}, girth~\cite{WZS2013}, or size~\cite{Liu-Liu2009}, and very recently for the families of outerplanar graphs~\cite{GBT2022, LLLW2022} and planar graphs~\cite{LLLW2022}. We note that the spreads of other matrices associated with a graph have also been extensively studied (see e.g. references in \cite{GBT2022, Cao-Vince93, Cvetkovic-Rowlinson90}).

Given two graphs $G$ and $H$, the \textit{join} of $G$ and $H$, denoted by $G\vee H$, is the graph obtained from the disjoint union of $G$ and $H$ by connecting every vertex of $G$ with every vertex of $H$. Let $P_k$ denote the path on $k$ vertices. Given two graphs $G$ and $H$, let $G\cup H$ denote the disjoint union of $G$ and $H$. Given a graph $G$ and a positive integer $k$, we use $kG$ to denote the disjoint union of $k$ copies of $G$. Given $v \subseteq V(G)$, let $N_G(v)$ denote the set of neighbors of $v$ in $G$, and let $d_G(v)$ denote the degree of $v$ in $G$, i.e., $d_G(v) = |N(v)|$. Given $S\subseteq V(G)$, define $N_G(S)$ as $N_G(S) = \cup_{v\in S} (N_G(v)\backslash S)$. Given a graph $G$ and disjoint vertex subsets $S,T\subseteq V(G)$, we use $E_G(S)$ to denote the set of edges in $E(G[S])$, and use $E_G(S,T)$ to denote the set of edges with one endpoint in $S$ and the other endpoint in $T$. For all above definitions, we may omit the subscript $G$ when there is no ambiguity. A graph $H$ is called a \textit{minor} of a graph $G$ if a graph isomorphic to
$H$ can be obtained from a subgraph of G by contracting edges. A graph $G$ is called \textit{$H$-minor-free} if $H$ is not a minor of $G$. 

There has been extensive work on finding the maximum spectral radius of $K_{s, t}$-minor-free graphs. Nikiforov~\cite{Nikiforov2017} showed that every sufficiently large $n$-vertex $K_{2,t}$-minor-free graph $G$ satisfies $\lambda_1(G)\leq (t-1)/2+\sqrt{n+(t^2-2t-3)/4}$, with equality if and only if $n \equiv 1 \pmod t$ and $G$ is $K_1\vee \floor{n/t}K_t$.
Tait~\cite{Tait2019} extended Nikiforov's result to $K_{s,t}$-minor-free graphs by giving an upper bound on the maximum spectral radius of a sufficiently large $n$-vertex $K_{s,t}$-minor-free graph $G$, and showed that the upper bound is tight if and only if $n \equiv s-1 \pmod t$ and $G$ is $K_{s-1}\vee \floor{(n-s+1)/t}K_t$. In the same paper, Tait conjectured that for all $t\geq s\geq 2$, the maximum spectral radius of a sufficiently large $n$-vertex $K_{s,t}$-minor-free graph is attained by $K_{s-1}\vee (pK_t \cup K_q)$, where $p, q$ satisfy that $n-s+1 = pt +q$ and $q\in [t]$.
Very recently, the $K_{s,t}$-minor-free graphs with maximum spectral radius were determined for $t\ge s\ge 2$ by Zhai and Lin~\cite{ZL2022}.

In this paper, we determine the maximum-spread $K_{2, t}$-minor-free graphs on $n$ vertices for sufficiently large $n$ and for all $t\geq 2$.

\begin{theorem}\label{thm:main}
For $t\geq 2$  and $n$ sufficiently large, the graph that maximizes the spread over the family of $K_{2,t}$-minor-free graphs on $n$ vertices is 
$$K_1\vee  \left( \left\lfloor \frac{2n+\xi_t}{3t} \right\rfloor K_t \cup \left(n-1- t\left\lfloor \frac{2n+\xi_t}{3t} \right\rfloor\right)  P_1\right)$$
where
$$\xi_t=\begin{cases}
2\left\lfloor \frac{3t}{4}-1 - \frac{(t-1)^2}{9}\right \rfloor  &  \mbox{ if } t \mbox{ is even},\\
\left\lfloor \frac{3t}{2}-2 - \frac{2(t-1)^2}{9}\right \rfloor & \mbox{ if } t\geq 3, \mbox{ and t is odd.}\\
\end{cases}
$$
The extremal graph is unique unless $t\equiv 4 \pmod {12}$ and
$\frac{2n+ \xi_t} {3t}$
is an integer.
In this special case, the maximum spread is achieved by two extremal graphs $$K_1\vee  \left( \left\lfloor \frac{2n+\xi_t}{3t} \right\rfloor K_t \cup \left(n-1- t\left\lfloor \frac{2n+\xi_t}{3t} \right\rfloor\right)  P_1\right)$$
and 
$$K_1\vee  \left( \left(\left\lfloor \frac{2n+\xi_t}{3t} \right\rfloor-1\right) K_t \cup \left(n-1- t\left(\left\lfloor \frac{2n+\xi_t}{3t} \right\rfloor-1\right)\right)  P_1\right).$$
\end{theorem}

We give a list of values of $\xi_t$ for small $t$ in Table~\ref{tab:ct}. 

\begin{table}[hbt]
    \centering
    \begin{tabular}{|c|c|c|c|c|c|c|c|c|c|c|c|c|c|c|c|c|c|c|c|}
    \hline
    \hline
$t$  & 2 & 3 & 4 & 5 & 6 & 7 & 8 & 9 & 10 & 11 & 12 & 13 & 14 & 15 & 16 & 17 & 18 & 19 & 20\\
\hline
$\xi_t$& 0 & 1 & 2 & 1 & 0 & 0 &-2 & -3 & -6 &-8 &-12 & -15 &-20 & -24 & -28 & -34 & -40 & -46 & -54 \\
\hline
    \end{tabular}
    \caption{The values of $\xi_t$ for $2\leq t \leq 20$.}
    \label{tab:ct}
\end{table}

Our paper is organized as follows. In Section \ref{sec:notation_lemmas}, we recall some useful lemmas and prove that in any maximum-spread $K_{2, t}$-minor-free graph $G$, there is a vertex $u_0$ which is adjacent to all other vertices in $G$. In Section \ref{sec:main_thm}, we show that $G - u_0$ is a disjoint union of cliques on $t$ vertices and isolated vertices and complete the proof of Theorem~\ref{thm:main}. 

\section{Notations and lemmas}\label{sec:notation_lemmas}
We first recall a result of Chudnovsky, Reed and Seymour~\cite{CRS2011} on the maximum number of edges of a $K_{2,t}$-minor-free graph, which extends an earlier result of Myers \cite{Myers2008}.
\begin{theorem}\cite{CRS2011}\label{thm:CRS_extremal}
Let $t \ge 2$ be a positive integer, and $G$ be a graph on $n>0$ vertices with no $K_{2,t}$ minor. Then
$$|E(G)|\leq \frac{1}{2}(t+1)(n-1).$$
\end{theorem}




Let $G$ be a graph which attains the maximum spread among all $n$-vertex $K_{2,t}$-minor-free graphs. 
As a first step towards proving Theorem \ref{thm:main}, we want to show that $G$ must contain a vertex of degree $n-1$. 

Recall the result of Nikiforov \cite{Nikiforov2017} on the maximum spectral radius of $K_{2,t}$-minor-free  graphs.

\begin{theorem}\cite{Nikiforov2017}\label{thm:nikiforov}
Let $t\geq 3$ and $G$ be a graph of order $n$ with no $K_{2,t}$ minor. If $n\geq 400t^6$, then the spectral radius $\lambda_1(G)$ satisfies
$$\lambda_1(G)\leq \frac{t-1}{2}+\sqrt{n+\frac{t^2-2t-3}{4}},$$
with equality if and only if $n\equiv 1\pmod{t}$ and $G=K_1\vee \lfloor n/t\rfloor K_t$.
\end{theorem}

We first give some upper and lower bounds on $\lambda_1(G)$ and $|\lambda_n(G)|$ when $n$ is sufficiently large. We use known expressions for the eigenvalues of a join of two regular graphs~\cite[pg.19]{BH2012}.

\begin{lemma}\cite{BH2012}\label{lem:joinreglemma}
Let $G$ and $H$ be regular graphs with degrees $k$ and $\ell$ respectively. Suppose that $|V(G)| = m$ and $|V(H)| = n$. Then, the characteristic polynomial of $G\vee H$ is $p_{G\vee H}(t) = ((t-k)(t-\ell)-mn)\frac{p_G(t)p_H(t)}{(t-k)(t-\ell)}$. In particular, if the eigenvalues of $G$ are $k = \lambda_1 \ge \ldots \ge \lambda_m$ and the eigenvalues of $H$ are $\ell = \mu_1 \ge \ldots \geq \mu_n$, then the eigenvalues of $G\vee H$ are $\{\lambda_i: 2\le i\le m\} \cup \{\mu_j: 2\le j\le n\} \cup \{x: (x-k)(x-\ell)-mn = 0\}$. 
\end{lemma}

We will apply Lemma~\ref{lem:joinreglemma} to the graph  $K_1\vee qK_t$ to obtain a lower bound on $S(G)$.

\begin{lemma}\label{lem:lambdan}
Let $G$ be a graph which attains the maximum spread among all $n$-vertex $K_{2,t}$-minor-free graphs. Then 
\[\sqrt{n-1} - \frac{t-1}{2}-O\left(\frac{1}{\sqrt{n}}\right) \le -\lambda_n(G) \le \lambda_1(G) \le \sqrt{n-1} +\frac{t-1}{2}+O\left(\frac{1}{\sqrt{n}}\right)  .\]
\end{lemma}

\begin{proof}
The upper bound of $\lambda_1(G)$ is due to Theorem \ref{thm:nikiforov}. Now let us prove the lower bound.
We will compute $S(K_1\vee qK_t)$, where $q = \floor{(n-1)/t}$. Note that $K_1\vee qK_t$ is $K_{2,t}$-minor-free. Hence, we can lower bound $S(G)$ by $S(K_1\vee qK_t)$.
By Lemma~\ref{lem:joinreglemma}, both $\lambda_1(K_1\vee qK_t)$ and $\lambda_n(K_1\vee qK_t)$ satisfy
the equation
$$\lambda(\lambda-(t-1))-qt=0.$$
Thus, we have
\begin{align*}
    \lambda_1(K_1\vee qK_t) &= \frac{t-1}{2}+\sqrt{qt+\frac{t^2-2t+1}{4}},\\
    \lambda_n(K_1\vee qK_t)&= \frac{t-1}{2}-\sqrt{qt+\frac{t^2-2t+1}{4}}.
\end{align*}
Thus $S(K_1\vee qK_t)=\sqrt{4qt+t^2-2t+1}$. Since $q=\lfloor (n-1)/t\rfloor$, we then have
$$S(G)\geq \sqrt{4qt+t^2-2t+1}\geq \sqrt{4(n-t)+t^2-2t+1}=\sqrt{4n+t^2-6t+1}=2\sqrt{n-1}+O\left(\frac{1}{\sqrt{n}}\right).
$$
Therefore, 
\begin{align*}
-\lambda_n(G) &= S(G)-\lambda_1(G) \\
&\geq 2\sqrt{n-1}+O\left(\frac{1}{\sqrt{n}}\right) - \left(\sqrt{n-1} +\frac{t-1}{2}+O\left(\frac{1}{\sqrt{n}}\right) \right)\\ &=\sqrt{n-1} - \frac{t-1}{2}-O\left(\frac{1}{\sqrt{n}}\right).
\end{align*}
\end{proof}

For the rest of this paper, let $\lambda_1\geq \cdots \geq \lambda_n$ be the eigenvalues of the adjacency matrix $A(G)$ of $G$. Given a vector $\cw \in \mathbb{R}^n$, let $\cw'$ denotes its transpose, and for each $i\in [n]$, let $\cw_i$ denote the $i$-th coordinate of $\cw$. 
Using the Rayleigh quotient of symmetric matrices, we have the following equalities for $\lambda_1$ and $\lambda_n$:
    \begin{align}
        \label{Rayleigh1}
        \lambda_1 &= \max_{\substack{\cw \in \mathbb{R}^n\\\cw\neq 0}} \frac{\cw' A(G) \cw}{\cw'\cw} = \max_{\substack{\cw \in \mathbb{R}^n\\\cw\neq 0}} \frac{2\sum_{ij\in E(G)} \cw_i \cw_j}{\cw'\cw}, \\
        \label{Rayleighn}
         \lambda_n &= \min_{\substack{\cw \in \mathbb{R}^n\\\cw\neq 0}}  \frac{\cw' A(G) \cw}{\cw'\cw} = \min_{\substack{\cw \in \mathbb{R}^n\\\cw\neq 0}} \frac{2\sum_{ij\in E(G)} \cw_i \cw_j}{\cw'\cw}.
    \end{align}

Let ${\bf x}$ and ${\bf z}$ be the eigenvectors of $A(G)$ corresponding to the eigenvalues $\lambda_1$ and $\lambda_n$ respectively. For convenience, let ${\bf x}$ and ${\bf z}$ be indexed by the vertices of $G$. By the Perron-Frobenius theorem, we may assume that all entries of ${\bf x}$ are positive. We also assume that $\cx$ and $\cz$ are normalized so that the maximum absolute values of the entries of $\cx$ and $\cz$ are equal to $1$, and so
there are vertices $u_0$ and $w_0$ with ${\bf x}_{u_0} = \cz_{w_0} = 1$. 

Let $V_+=\{v\colon {\bf z}_v> 0\}$, $V_0=\{v\colon {\bf z}_v= 0\}$,
and $V_-=\{v\colon {\bf z}_v < 0\}$. 
Since $\cz$ is a non-zero vector, at least one of $V_{+}$ and $V_{-}$ is non-empty. By considering the eigen-equations of $\lambda_n \sum_{v\in V_{+}} \cz_v$ or $\lambda_n \sum_{v\in V_{-}} \cz_v$, we obtain that both $V_{+}$ and $V_{-}$ are non-empty.
For any vertex subset $S$, we define the \textit{volume} of $S$, denoted by $\vol(S)$, as 
$\vol(S)= \sum_{v\in S} |\cz_v|$. 
In the following lemmas, we use the bounds of $\lambda_n$ to deduce some information on $V_{+}$, $V_{-}$ and $V_0$.

\begin{lemma}
We have
$$\vol(V(G))=O(\sqrt{n}).$$
\end{lemma}
\begin{proof}
For any vertex $v \in V(G)$, we have
$$d(v) \geq |\sum_{y\in N(v)}z_y|=|\lambda_n| |z_v|.$$
Applying Theorem \ref{thm:CRS_extremal}, we have
$$(t+1)n\geq \sum_{v\in V}d(v) \geq \sum_{v\in V(G)}  |\lambda_n| |z_v|
=|\lambda_n| \vol(V).$$
By Lemma \ref{lem:lambdan}, $|\lambda_n|\geq \sqrt{(n-1)} - \frac{t-1}{2}-O\left(\frac{1}{\sqrt{n}}\right)$.
We thus have $\vol(V)=O(\sqrt{n})$.
\end{proof}

\begin{lemma}\label{lem:almostfulldegree}
There exists some constant $C_1$ such that for all $n$ sufficiently large, we have
\begin{enumerate}
    \item $d(w_0)\geq n- C_1\sqrt{n}$.
    \item For any vertex $u\not=w_0$, $d(u)\leq 2C_1\sqrt{n}$ and $|z_u|=O(\frac{1}{\sqrt{n}})$.
\end{enumerate}
\end{lemma}
\begin{proof}
For any $u\in V_+$, we have
$$|\lambda_n| z_u =-\lambda_n z_u =  -\sum_{v \in N(u)} z_v \leq \sum_{v \in N(u)\cap V_-} |z_v|.$$
Therefore, for any $u\in V_{+}$,
\begin{align*}
|\lambda_n|^2 z_u \leq \sum_{v \in N(u)\cap V_-} |\lambda_n| |z_v|
& = \sum_{v \in N(u)\cap V_-} \lambda_n z_v\\
&\leq  \sum_{v \in N(u)\cap V_-} \sum_{y\in N(v)\cap V_+} z_y \\
&\leq d(u) z_u + \sum_{y\in V_+\setminus \{u\}} z_y |N(y)\cap N(u)\cap V_-|\\
&\leq  d(u) z_u + \sum_{y\in V_+\setminus \{u\}} z_y (t-1) \quad \textrm{since $G$ is $K_{2,t}$-minor-free}\\
&\leq  d(u) z_u + (t-1)\vol(V_+).
\end{align*}
Similarly, if $u\in V_-$, we have
$$ |\lambda_n|^2 |z_u| \leq   d(u) |z_u| + (t-1)\vol(V_-). $$
Setting $u=w_0$, we get
$$|\lambda_n|^2-d(w_0)\leq  (t-1)\vol(V)
=O(\sqrt{n}).$$
Hence,
$$d(w_0)\geq n - O(\sqrt{n}) \geq n -C_1\sqrt{n}, \textrm{ for some $C_1 > 0$}.$$ 
Now we show $d(u)\leq 2C_1\sqrt{n}$ for any vertex $u$ other than $w_0$.
Otherwise, if $d(u)\geq 2C_1\sqrt{n}$, then
$u$ and $w_0$ have at least $C_1\sqrt{n}\geq t$ neighbors (when $n$ is sufficiently large).
Thus $G$ contains the subgraph $K_{2,t}$, contradicting that $G$ is $K_{2,t}$-minor-free.
It then follows that for all $u\not= w_0$, we have
$$|z_u|\leq \frac{(t-1)\vol(V)}{|\lambda_n|^2 -d(u)} = O\left(\frac{1}{\sqrt{n}}\right).$$
\end{proof}

\begin{lemma}
We have
\begin{enumerate}[(i)]
    \item $u_0 = w_0$.
    \item For any vertex $v\neq w_0$, $\cx_v = O\lp \frac{1}{\sqrt{n}} \rp$.
\end{enumerate}
\end{lemma}
\begin{proof}
    We will prove (ii) first. For any $v \in V(G)\backslash \{w_0\}$, we have
    \begin{align}
      \lambda_1^2 x_v &= \lambda_1 \dss_{s\in N(v)} \cx_s \nonumber\\
                      &\leq \lambda_1 \lp \cx_{w_0} + \dss_{s \in N(v) \backslash \{w_0\}} \cx_s\rp \nonumber\\
                      &\leq \lambda_1 + \dss_{s \in N(v) \backslash \{w_0\}} \dss_{t\in N(s)} \cx_t \nonumber\\
                      & \leq \lambda_1 + \dss_{s \in N(v) \backslash \{w_0\}} \lp \cx_{w_0} + \dss_{t\in N(s)\backslash\{w_0\}} \cx_t\rp \nonumber\\
                      & \leq \lambda_1 + (2C_1\sqrt{n})\cx_{w_0} + \dss_{s \in N(v) \backslash \{w_0\}} \dss_{t\in N(s)\backslash\{w_0\}} \cx_t \label{ieq:expansion}
    \end{align}
\begin{claim}
    For any $v\in V(G)\backslash \{w_0\}$, we have
    $\dss_{s \in N(v) \backslash \{w_0\}} \dss_{t\in N(s)\backslash\{w_0\}} \cx_t = O(\sqrt{n})$.
\end{claim}
\begin{proof}
    Observe that 
    \begin{align}
    \dss_{s \in N(v) \backslash \{w_0\}} \dss_{t\in N(s)\backslash\{w_0\}} \cx_t 
        & \leq \dss_{s \in N(v) \backslash \{w_0\}} \dss_{t\in N(s)\backslash\{w_0\}} 1  \nonumber\\
        &= \abs{\{(s,t)\in V(G)^2: s\in N(v)\backslash \{w_0\}, t\in N(s)\backslash \{w_0\}\}}. \nonumber\\
        &\leq 2|E_{G-w_0}(N(v))| + |E_{G-w_0}(N(v), V(G)\backslash N(v))| \nonumber\\
        & \leq 2|E_{G-w_0}(N(v))|+  |E_{G-w_0}(N(v), N_G(w_0)\backslash N(v))| +\nonumber \\
         & \quad\quad |E_{G-w_0}(N(v), V(G)\backslash (N_G(w_0)\cup N(v))| \label{eq:ST-edges}
    \end{align}
By Theorem \ref{thm:CRS_extremal} and Lemma \ref{lem:almostfulldegree},
$$2|E_{G-w_0}(N(v))|\leq (t+1)2C_1\sqrt{n}.$$
Since $G$ is $K_{2,t}$-minor-free, the bipartite graph induced by $E_{G-w_0}(N(v), N_G(w_0)\backslash N(v))$ is $K_{1,t}$-free. Hence every vertex in $N(v)$ has at most $t-1$ neighbors in $N_G(w_0)\backslash N(v)$. It follows that 
$$|E_{G-w_0}(N(v), N_G(w_0)\backslash N(v))| \leq (t-1)|N(v)| \leq 2(t-1)C_1\sqrt{n}.$$
Similarly, every vertex in $V(G)\backslash (N_G(w_0)\cup N(v))$ has at most $t-1$ neighbors in $N(v)$. It follows that 
$$|E_{G-w_0}(N(v), V(G)\backslash (N_G(w_0)\cup N(v))| \leq (t-1) |V(G)\backslash (N_G(w_0)\cup N(v))| \leq (t-1)C_1\sqrt{n}.$$
Hence by \eqref{eq:ST-edges},
$$ \dss_{s \in N(v) \backslash \{w_0\}} \dss_{t\in N(s)\backslash\{w_0\}} \cx_t \leq (t+1)2C_1\sqrt{n} + 2(t-1)C_1\sqrt{n}+(t-1)C_1\sqrt{n} = O(\sqrt{n}).$$
\end{proof}
Now by the claim above and \eqref{ieq:expansion}, we have that
$$\lambda_1^2 x_v \leq \lambda_1 + 2C_1\sqrt{n} + O(\sqrt{n}) = O(\sqrt{n}).$$
Using the fact that $|\lambda_1|\geq \sqrt{(n-1)} - \frac{t-1}{2}-O\left(\frac{1}{\sqrt{n}}\right)$, we have that 
$$x_v = O\left(\frac{1}{\sqrt{n}}\right).$$
It follows that $w_0 = u_0$.
\end{proof}

\begin{lemma}\label{lem:max-degree}
We have that $d(u_0) = n-1$.
\end{lemma}
\begin{proof}
Suppose for contradiction that $d(u_0)<n-1$. Let $S=V(G)\backslash (N(u_0)\cup \{u_0\})$. Then $S\neq \emptyset$. By Lemma \ref{lem:almostfulldegree}, $|S|\leq C_1 \sqrt{n}$. Note that $G[S]$ is also $K_{2,t}$-minor-free. Hence by Theorem \ref{thm:CRS_extremal}, $|E(G[S])|\leq \frac{1}{2}(t+1)|S|$. It follows that there exists a vertex $v \in S$ such that $d_S(v) \leq t+1$. Moreover, since $G$ is $K_{2,t}$-minor-free, we have that $d_{N(u_0)}(v) \leq t-1$. Hence $d_G(v) \leq t+1 +(t-1) = 2t$.
Let $G'$ be obtained from $G$ by removing all the edges of $G$ incident with $v$ and adding the edge $v u_0$.

We claim that $\lambda_n(G') < \lambda_n(G)$. Indeed, consider the vector $\tilde{\cz}$ such that $\tilde{\cz}_{u} = \cz_u$ for $u\neq v$ and $\tilde{\cz}_v = -|\cz_v|$. Then for sufficiently large $n$, we have 
    \begin{align*}
        \tilde{\cz}' A(G') \tilde{\cz} &\leq \cz' A(G) \cz + 2 \dss_{y\sim v} |\cz_y \cz_v| - 2|\cz_v| z_{u_0}\\
                      & \leq \cz' A(G)\cz + 2 \cdot 2t \cdot O\left(\frac{1}{\sqrt{n}}\right)\cdot |\cz_v| -  2 |\cz_v|\\
                      & < \cz' A(G) \cz.
    \end{align*}   
By the Rayleigh quotient, we have 
$$\lambda_n(G') \leq \frac{\tilde{\cz}' A(G') \tilde{\cz}}{\tilde{\cz}' \tilde{\cz}} < \frac{ \cz' A(G) \cz}{ \cz' \cz} = \lambda_n(G).$$
Similarly, we claim that $\lambda_1(G') > \lambda_1(G)$. Indeed,
     \begin{align*}
     \cx' A(G') \cx   &= \cx' A(G) \cx - 2 \dss_{y\sim v} \cx_y \cx_v + 2\cx_v \cx_{u_0}\\
                      & \geq \cx' \lambda_1(G) \cx - 2 \cdot 2t \cdot O\left(\frac{1}{\sqrt{n}}\right)\cdot \cx_v +  2\cx_v\\
                      & > \cx' A(G)\cx.
    \end{align*}
Using the Rayleigh quotient again, 
$$\lambda_1(G') \geq \frac{ \cx' A(G') \cx }{ \cx' \cx } > \frac{ \cx' A(G) \cx }{ \cx'\cx } = \lambda_1(G).$$
Therefore, we have $S(G') =\lambda_1(G') -\lambda_n(G') > \lambda_1(G) -\lambda_n(G) = S(G)$, giving a contradiction.
\end{proof}

\section{Proof of Theorem \ref{thm:main}}\label{sec:main_thm}

By Lemma \ref{lem:max-degree}, a maximum-spread $K_{2,t}$-minor-free graph $G$ has a vertex $u_0$ with degree $n-1$. 
Let $\alpha$ be a normalized eigenvector corresponding to an eigenvalue $\lambda$ of the adjacency matrix of $G$ so that $\alpha_{u_0}=1$. Let $H=G - u_0$ and $A_H$ be the adjacency matrix of $H$. Note that $H$ is $K_{1,t}$-minor-free since $G$ is $K_{2,t}$-minor-free.
Let $I$ denote the identity matrix of dimension $n-1$ and let $\one$ denote the all one vector of dimension $n-1$.
Moreover, let $\bf x$ denote the restriction of $\alpha$ to the vertices of $H$.
The following lemma computes the vector ${\bf x}$.

\begin{lemma}\label{l1}
  We have
    \begin{equation} \label{eq:taylor_eq} 
        {\bf x} =\sum_{k=0}^\infty \lambda^{-(k+1)} A_H^k \one.
    \end{equation}
\end{lemma}
\begin{proof}
Since $H$ is $K_{1,t}$-minor-free, the maximum degree of $H$ is at most $t-1$. For sufficiently large $n$,
both $\lambda_1(G)$ and $|\lambda_n(G)|$ are greater than $t-1$. 
Each vertex $v\not=u_0$ is adjacent to $u_0$ and $\alpha_{u_0}=1$. Hence when restricting the coordinates of $A(G)\alpha$ to $V(G)\backslash\{u_0\}$, we have that
\begin{equation}\label{eq:bx}
    A_H\bx + \one = \lambda \bx.
\end{equation}
It then follows that
\begin{align}
    \bx &= (\lambda I-A_H)^{-1}\one  \nonumber\\
    &=\lambda^{-1}  (I-\lambda^{-1}A_H)^{-1}\one  \nonumber\\
    &= \lambda^{-1} \sum_{k=0}^\infty (\lambda^{-1}A_H)^{k} \one\nonumber\\
    &= \sum_{k=0}^\infty \lambda^{-(k+1)} A_H^k \one. 
    \end{align}
Here we use the assumption that $|\lambda|> t-1 \geq \lambda_1(A_H)$ so that the infinite series converges.
\end{proof}

\begin{lemma}\label{l8}
Both $\lambda_1$ and $\lambda_n$ satisfy the following equation.
\begin{equation} \label{eq:lambda}
        \lambda^2 =  (n-1) +
  \sum_{k=1}^\infty \lambda^{-k} \one' A_H^k \one.  
\end{equation}
\end{lemma}
\begin{proof}
The eigen-equation at $u_0$ gives
\begin{equation} \label{eigen_u}
    \lambda  =\lambda \cx_{u_0}=  \sum_{v\in V(H)} \cx_v.
\end{equation}
Applying Lemma \ref{l1}, we get
\begin{align*}
 \sum_{v\in V(H)} \cx_v &= \one' \cdot \bx\\
 &=\one' \cdot \sum_{k=0}^\infty \lambda^{-(k+1)} A_H^k \one \\
 &=   \sum_{k=0}^\infty \lambda^{-(k+1)} \one' A_H^k \one. 
\end{align*}
Plugging it to Equation \eqref{eigen_u}, we have 
\begin{equation} \label{eigen_u2}
    \lambda =    (n-1)\frac{1}{\lambda} +
  \sum_{k=1}^\infty \lambda^{-(k+1)} \one' A_H^k \one.  
\end{equation}
Multiplying by $\lambda$ on both sides, we get Equation \eqref{eq:lambda}.
\end{proof}

For $k=1,2,3\ldots$, let $a_k=  \one' A_H^k \one$. In particular, $a_1= \one' A_H \one=\sum_{v\in V(H)} d_H(v)=2|E(H)|$;
$a_2= \one' A_H^2 \one=\sum_{v\in V(H)} d_H(v)^2$. 

\begin{lemma} \label{lem:approximation}
We have the following estimation of the spread of $G$:
\begin{equation}\label{eq:spreadapprox}
    S(G)=2\sqrt{n-1}+\frac{2c_2}{\sqrt{n-1}} + \frac{2c_4}{(n-1)^{3/2}} +  \frac{2c_6}{(n-1)^{5/2}} + O\left(n^{-7/2}\right).
\end{equation}
Here 
\begin{align}
c_2 &= -\frac38 \left(\frac{a_1}{n-1}\right)^2 + \frac12 \frac{a_2}{n-1}, \label{eq:c2}
\\
c_4 &=-\frac{105}{128} \left(\frac{a_1}{n-1}\right)^4 +\frac{35}{16} \left(\frac{a_1}{n-1}\right)^2\frac{a_2}{n-1}
-\frac{5}{8}\left(\frac{a_2}{n-1}\right)^2 -\frac{5}{4}\frac{a_1}{n-1}\frac{a_3}{n-1} +\frac{1}{2} \frac{a_4}{n-1} \label{eq:c4}\\
c_6&=-\frac{3003}{1024} \left(\frac{a_1}{n-1}\right)^6 +\frac{3003}{256} \left(\frac{a_1}{n-1}\right)^4\frac{a_2}{n-1}
-\frac{693}{64} \left(\frac{a_1}{n-1}\right)^2\left(\frac{a_2}{n-1}\right)^2
+\frac{21}{16}\left(\frac{a_2}{n-1}\right)^3  
\nonumber\\
&\hspace*{5mm} 
-\frac{21}{32}\left(11\left(\frac{a_1}{n-1}\right)^3
-12\left(\frac{a_1}{n-1}\right)\left(\frac{a_2}{n-1}\right)
\right )\left(\frac{a_3}{n-1}\right)
- \frac{7}{8}
\left(\frac{a_3}{n-1}\right)^2 \nonumber\\
&\hspace*{5mm} 
+
\frac{7}{16}\left(9\left(\frac{a_1}{n-1}\right)^2 -4\frac{a_2}{n-1}\right)\frac{a_4}{n-1}
- \frac{7}{4}\frac{a_1}{n-1}\frac{a_5}{n-1} +\frac{1}{2} \frac{a_6}{n-1}. \label{eq:c6}
\end{align}
\end{lemma}
\begin{proof}
    Recall that by \eqref{eq:lambda}, we have that for $\lambda \in \{\lambda_1, \lambda_n\}$,
    $$ \lambda =    (n-1)\frac{1}{\lambda} +
  \sum_{k=1}^\infty \lambda^{-(k+1)} \one' A_H^k \one.  $$
Multiplying by $\lambda$ on both sides, we have that
\begin{equation}
    \lambda^2 = (n-1) +  \dss_{k=1}^{\infty} \frac{a_k}{\lambda^k}.
\end{equation}
  By similar logic in the main lemma of the appendix in \cite{LLLW2022}, $\lambda$ has the following series expansion:
 $$\lambda_1 = \sqrt{(n-1)} + c_1 + \frac{c_2}{\sqrt{n-1}} + \frac{c_3}{n-1} + \frac{c_4}{(n-1)^{\frac{3}{2}}} + \frac{c_5}{(n-1)^2} + \frac{c_6}{(n-1)^\frac{5}{2}} + O\left(n^{-7/2}\right).$$
Similarly, 
  $$\lambda_n = -\sqrt{(n-1)} + c_1 - \frac{c_2}{\sqrt{n-1}} + \frac{c_3}{n-1} - \frac{c_4}{(n-1)^{\frac{3}{2}}} + \frac{c_5}{(n-1)^2} - \frac{c_6}{(n-1)^\frac{5}{2}} + O\left(n^{-7/2}\right).$$

Using SageMath (computation available at \href{https://github.com/wzy3210/graph_spreads}{https://github.com/wzy3210/graph\_spreads}), we get that $c_2, c_4, c_6$ are the values in Equations \eqref{eq:c2}, \eqref{eq:c4}, \eqref{eq:c6} respectively.
It follows that
$$S(G) =\lambda_1 - \lambda_n =  2\sqrt{(n-1)} + \frac{2c_2}{\sqrt{n-1}} + \frac{2c_4}{(n-1)^{\frac{3}{2}}} + \frac{2c_6}{(n-1)^\frac{5}{2}} + O\left(n^{-7/2}\right).$$
\end{proof}

\begin{lemma}\label{l:k2t}
For sufficiently large $n$, a maximum-spread $K_{2,t}$-minor-free $n$-vertex graph $G$ must be of the form
$$K_1\vee \left(\ell K_t\cap (n-1-\ell t)P_1\right).$$
\end{lemma}
\begin{proof}[Proof of Lemma~\ref{l:k2t}]
  By Lemma~\ref{lem:max-degree}, there exists a vertex $u_0 \in V(G)$ of degree $n-1$. Let $H = G-u_0$. 
Since $G$ is $K_{2,t}$-minor-free, every vertex in $H$ has at most $t-1$ neighbors in $H$. Thus $\Delta(H) \leq t-1$, and it follows that
$$a_2= \one' A_H^2 \one=\sum_{v\in V(H)} d_H(v)^2 \leq (t-1) \sum_{v\in V(H)} d_H(v) = (t-1) a_1.$$
Note that $$a_1 = 2|E(H)| \leq \Delta(H)|V(H)| \leq (t-1)(n-1).$$
It follows that $a_i\leq (t-1)^i(n-1)$ for all $i\geq 2$. 
By Lemma \ref{lem:approximation}, we have the following estimation of the spread of $G$:
\begin{equation}
    S(G)=2\sqrt{n-1}+\frac{2c_2}{\sqrt{n-1}} + \frac{2c_4}{(n-1)^{3/2}} +  \frac{2c_6}{(n-1)^{5/2}} + O\left(n^{-7/2}\right),
\end{equation}
where $c_2, c_4, c_6$ are computed in Lemma \ref{lem:approximation}, and all $c_i$s are bounded by constants depending on $t$.
Note
\begin{align*}
    c_2 &= -\frac38 \left(\frac{a_1}{n-1}\right)^2 + \frac12 \frac{a_2}{n-1}\\
        &\leq -\frac38 \left(\frac{a_1}{n-1}\right)^2 + \frac12 \frac{(t-1)a_1}{n-1}\\
        &= \frac{(t-1)^2}{6} -\frac{3}{8} \lp \frac{a_1}{n-1} -\frac{2}{3}(t-1)\rp^2 \\
        &\leq \frac{(t-1)^2}{6},
\end{align*}
where in the last inequality, the equality is only achieved when $a_1 = \frac{2}{3}(t-1)(n-1)$. For $G_0=K_1\vee  \left( \left\lfloor \frac{2n+\xi_t}{3t} \right\rfloor K_t \cup \left(n-1- t\left\lfloor \frac{2n+\xi_t}{3t} \right\rfloor\right)  P_1\right)$, we have $\frac{a_1}{n-1}= \frac{2}{3}(t-1)+O\lp \frac{1}{n}\rp$.
Thus
$$S(G_0) = 2\sqrt{n-1} + \frac{(t-1)^2}{3\sqrt{n-1}} + O\lp\frac{1}{n^{3/2}}\rp.$$
\begin{claim}\label{cl:Asmall}
There exists a constant $C>0$ such that the value of $a_1$ that maximizes $S(G)$ lies in the interval $(\frac{2}{3}(t-1)(n-1)- Cn^{1/2}, \frac{2}{3}(t-1)(n-1)+ Cn^{1/2})$.
\end{claim}
\begin{proof}
Let $C$ be a sufficiently large constant chosen later. Suppose for contradiction that $a_1$ is not contained in the interval above.
Then, we must have that
$$c_2 \leq \frac{(t-1)^2}{6} - \frac{3C^2n} {8(n-1)^2}.$$
This implies that
$$S(G) \leq 2\sqrt{n-1} + 2\cdot \frac{\frac{(t-1)^2}{6} - \frac{3C^2n} {8(n-1)^2}}{\sqrt{n-1}} + O\lp \frac{1}{(n-1)^{3/2}}\rp < S(G_0),$$
when $C$ is chosen to be large enough such that
$$-\frac{2\cdot 3C^2}{8(n-1)^2}\frac{n}{\sqrt{n-1}} + O\lp \frac{1}{(n-1)^{3/2}}\rp < 0.$$
This gives us a contradiction since $G$ is assumed to be an extremal graph that maximizes the spread over all $K_{2,t}$-minor-free graphs.
\end{proof}

From now on, we assume that $a_1 \in (\frac{2}{3}(t-1)(n-1)- Cn^{1/2}, \frac{2}{3}(t-1)(n-1)+ Cn^{1/2})$ for some constant $C>0$.

\begin{claim}\label{cl:Asmall'}
There is a constant $C_2$ such that the value of $a_2$ lies in the interval $[(t-1)a_1-C_2, (t-1)a_1]$.
\end{claim}
\begin{proof}
Let $C_2$ be a sufficiently large constant chosen later. Suppose for contradiction that $a_2<(t-1)a_1-C_2$. We then have that 
$$
S(G)\leq  2\sqrt{n-1} + \frac{(t-1)^2}{3\sqrt{n-1}} 
-\frac{C_2}{(n-1)^{3/2}}
+ O\lp\frac{1}{n^{3/2}}\rp< S(G_0),
$$
if we choose $C_2$ large enough, giving a contradiction.
\end{proof}

\begin{claim}\label{cl:Asmall2}
For  $i\geq 2$, we have $a_i\in [(t-1)^{i-1}(a_1-(i-1)C_2),a_1(t-1)^{i-1}] $.
\end{claim}
\begin{proof}
We will show this claim by inducting on $i\geq 2$. 
Note that by Claim \ref{cl:Asmall'}, we have that 
$a_2 \geq (t-1)a_1 -C_2$. Moreover, $a_2\leq (t-1)a_1$ since $\Delta(H)\leq t-1$. Hence the base case holds.
Moreover, we also obtain from above that $C_2 \geq (t-1)a_1 - a_2$.

Let $H'$ be the set of vertices in $H$ such that its degree is in the interval $[1,t-2]$. We have
$$C_2\geq (t-1)a_1-a_2=\sum_{v\in H'}(t-1-d(v))d(v)\geq (t-2)|H'|.$$
This implies 
$$|H'|\leq \frac{C_2}{t-2}.$$
For a vertex $v \in H$ and non-negative integer $k$, let $w_k(v)$ denote the number of walks of length $k$ in $H$ starting at $v$. Observe that
\begin{align*}
    (t-1)a_{i-1}-a_i &= (t-1)\dss_{v\in V(H)} w_{i-1}(v) -\dss_{v\in V(H)} w_i(v)\\
    &\leq  \dss_{v\in H'} \lp(t-1)-d_H(v)\rp (t-1)^{i-1}\\
    &\leq  |H'|(t-2)(t-1)^{i-1}  \\
    &\leq C_2(t-1)^{i-1}
\end{align*}
Thus,
\begin{align*}
    a_i&\geq (t-1)a_{i-1} -C_2(t-1)^{i-1}\\
     &\geq (t-1)((t-1)a_{i-2} -C_2(t-1)^{i-2}) -C_2(t-1)^{i-1} \hspace*{1cm}\mbox{ by induction}\\
     &=(t-1)^2a_{i-2}-2C_2(t-1)^{i-1}\\
     &\geq (t-1)^{i-1}a_{1} -(i-1)C_2(t-1)^{i-1},
\end{align*}
where the last inequality is obtained by repeatedly applying 
induction.
\end{proof}

\begin{claim}\label{cl:a1a2}
    $a_2 = (t-1) a_1$.
\end{claim}
\begin{proof}
    Assume that $a_1 = \frac{2}{3}(t-1)(n-1) + A$, and $a_2 = (t-1)a_1-B$, where $A \in [-Cn^{1/2}, Cn^{1/2}]$ and $0 \leq B \leq C_2$. For $i\geq 2$, let $c_i(G), c_i(G_0)$ denote the $c_i$ values of $G$ and $G_0$ respectively. Observe that
\begin{align*}
    c_2(G) &= -\frac38 \left(\frac{a_1}{n-1}\right)^2 + \frac12 \frac{a_2}{n-1}\\
          &= \frac{(t-1)^2}{6} - \frac{3A^2}{8(n-1)^2} - \frac{B}{2(n-1)}.
\end{align*}    
It follows that 
$$c_2(G)-c_2(G_0) = -\frac{3A^2}{8(n-1)^2} -\frac{B}{2(n-1)} +O(n^{-2}).$$
Moreover, by Claim \ref{cl:Asmall2}, for all $i\geq 4$, we have that
$$c_i(G) - c_i(G_0) = O(n^{-1/2}).$$
Thus 
\begin{align*}
    S(G) -S(G_0) &= 2 \cdot \frac{c_2(G)-c_2(G_0)}{\sqrt{n-1}}
+ 2 \cdot \frac{c_4(G)-c_4(G_0)}{(n-1)^{3/2}} +O((n-1)^{-5/2})\\
&\leq 2 \cdot \frac{ O\lp n^{-2}\rp
    - \frac{3A^2}{8(n-1)^2} - \frac{B}{2(n-1)}}{\sqrt{n-1}}
+ 2 \cdot\frac{O(n^{-1/2})}{(n-1)^{3/2}} + O((n-1)^{-5/2}).
\end{align*}
Since $S(G)\geq S(G_0)$, this implies that $A = O(n^{1/4})$, $B=0$ and thus $a_2 = (t-1)a_1$.
\end{proof}

\begin{claim}
    $H$ is the union of vertex disjoint $K_t$s and isolated vertices. 
\end{claim}
\begin{proof}
Recall that $a_1= \one' A_H \one=\sum_{v\in V(H)} d_H(v)=2|E(H)|$, and $a_2= \one' A_H^2 \one=\sum_{v\in V(H)} d_H(v)^2$. By Claim \ref{cl:a1a2}, we have that
$$\sum_{v\in V(H)} d_H(v)^2 = (t-1)\sum_{v\in V(H)} d_H(v).$$
Since $d_H(v) \leq t-1$ for every $v\in V(H)$, it follows that $H$ is the disjoint union of $(t-1)$-regular graphs and isolated vertices. Let $K$ be an arbitrary non-trivial component of $H$. We will show that $K$ is a clique on $t$ vertices.

We first claim that for any $u,v\in V(K)$, $N(u)\cap N(v) \neq \emptyset$. Otherwise, pick a shortest path $P$ between $u$ and $v$ in $P$. Observe that $|V(P) \cap N(u)| = |V(P)\cap N(v)| = 1$.
Contract $uPv$ into one vertex $x$ (call the new graph $G'$). Note that $x$ and $N(u)\cup N(v)$ form a $K_{1,t}$ in $G'$. Together with $u_0$ which is adjacent to every vetex in $K$, we have a $K_{2,t}$ minor in $G$, giving a contradiction.

Next, we claim that for any $u, v\in V(K)$ with $uv \notin E(K)$, $|N(u)\cap N(v)| \geq t-2$. Otherwise, $|N(u)\backslash N(v)| \geq 2$ and $|N(v)\backslash N(u)|\geq 2$. Similar to before, pick an arbitrary vertex $w \in N(u) \cap N(v)$ and contract the path $uwv$, we then obtain a $K_{1,t}$-minor in $K$, and thus a $K_{2,t}$-minor in $G$. Similarly, for any $u,v\in V(K)$ with $uv\in E(K)$, we have $|N(u)\cap N(v)| \geq t-3$.
Moreover, note that for any $u,v \in V(K)$, $|N(u) \cap N(v)| \leq t-2$, since otherwise $\{u,v\}$ and $(N(u)\cap N(v)) \cup \{u_0\}$ forms a $K_{2,t}$ in $G$, giving a contradiction.
Hence, we have that for any $u,v\in V(K)$ with $uv\notin E(K)$, $|N(u)\cap N(v)| = t-2$.

Now if $K$ is not a clique on $t$ vertices, then let $u,v\in V(K)$ be two vertices in $K$ such that $uv\not\in E(K)$. By the above claim, there exists $u', v' \in V(K)$ such that $u'\in N(u)\backslash N(v)$ and $v' \in N(v) \backslash N(u)$.

We claim that $u'v' \notin E(K)$. Indeed, if $u'v'\in E(K)$, contract $v'u'$ into $w'$. Then $\{u,v\} \cup \lp \{w', u_0\} \cup (N(u)\cap N(v)) \rp$ is a $K_{2,t}$ minor in $G$, giving a contradiction.
Now note that since $u'v \notin E(K)$, we have $|N(u')\cap N(v)| = t-2$. It follows that $N(u')\cap N(v) = N(u)\cap N(v)$. Similarly, $N(v') \cap N(u) = N(u)\cap N(v)$. 

We claim that each vertex in $N(u)\cap N(v)$ has exactly one non-neighbor in $N(u)\cap N(v)$. Indeed, let $w$ be an arbitrary vertex in $N(u)\cap N(v)$. Note that $w$ cannot be adjacent to all other vertices in $N(u)\cap N(v)$; otherwise since $w$ is adjacent to $u',u$, and $v'$, we then have $d(w)\geq (t-3)+3=t$, contradicting that $K$ is $(t-1)$-regular. On the other hand, suppose $w$ has at least two non-neighbors in $N(u)\cap N(v)$. Then it follows that $|N(w)\cap (N(u)\cap N(v))| \leq t-2 -3 = t-5$. Now observe that $N(u)\cap N(w) = (N(w)\cap N(u)\cap N(v)) \cup \{u'\}$. It follows that 
$|N(u)\cap N(w)| \leq t-4$, contradicting our claim before that any two adjacent vertices must have at least $t-3$ common neighbors. Hence $w$ has exactly one non-neighbor in $N(u)\cap N(v)$, say $w'$. But now observe that
$$N(w)\cap N(w') \supseteq (N(u)\cap N(v)\backslash \{w,w'\}) \cup \{u,u',v, v'\},$$
which implies that $|N(w)\cap N(w')| \geq t-4+4 = t$, contradicting that $K$ is $(t-1)$-regular.
Hence by contradiction, $K$ is a clique on $t$ vertices.
\end{proof}
This completes the proof of Lemma \ref{l:k2t}.
\end{proof}

\begin{proof}[Proof of Theorem \ref{thm:main} ]
For sufficiently large $n$, let $G$ be an extremal graph attaining the maximum spread among all $n$-vertex $K_{2,t}$-minor-free graphs. By Lemma \ref{l:k2t}, we only need to consider graphs in the form of $G_\ell = K_1\vee \left(\ell K_t \cup (n-1-\ell t)P_1 \right)$.
It also follows from Lemma \ref{l:k2t} that for $i\geq 1$, 
\begin{equation}
  a_i = \ell t(t-1)^i.
\end{equation}

For each $i\geq 2$, let $c_i(\ell)$ denote the $c_i$ value of $G_{\ell}$.  Plugging $a_i$'s into Equations \eqref{eq:c2}, \eqref{eq:c4}, and \eqref{eq:c6}, we get
\begin{align}
c_2(\ell)&= -\frac{3}{8}\frac{t^2(t-1)^2}{(n-1)^2}\ell^2 + \frac{1}{2} \frac{t(t-1)^2}{(n-1)}\ell,\\
c_4(\ell) &=  -\frac{105}{128}\frac{t^4(t-1)^4}{(n-1)^4}\ell^4 +\frac{35}{16} \frac{t^3(t-1)^4}{(n-1)^3}\ell^3 -\frac{15}{8}\frac{t^2(t-1)^4}{(n-1)^2}\ell^2 
+\frac{1}{2}\frac{t(t-1)^4}{(n-1)}\ell,\\
 c_6(\ell) &= -\frac{3003}{1024}\frac{t^6(t-1)^6}{(n-1)^6}\ell^6
    +\frac{3003}{256}\frac{t^5(t-1)^6}{(n-1)^5}\ell^5-\frac{1155}{64}\frac{t^4(t-1)^6}{(n-1)^4}\ell^4 \nonumber\\
    &\hspace*{7mm}+\frac{105}{8} \frac{t^3(t-1)^6}{(n-1)^3}\ell^3 -\frac{35}{8}\frac{t^2(t-1)^6}{(n-1)^2}\ell^2 
+\frac{1}{2}\frac{t(t-1)^6}{(n-1)}\ell.
\end{align}

Let $\ell_1= \frac{2(n-1)}{3t}$, which is the (possibly real) argmax value of $c_2(\ell)$. Let
$\ell_0=\lfloor \frac{2n+\xi_t}{3t}\rfloor$ be the target maximum integer point of $S(G_\ell)$. By Claim~\ref{cl:Asmall}, we assume that $\ell \in (\ell_1 - C\sqrt{n-1}, \ell_1 + C\sqrt{n-1})$. Let us compute $S(G_{\ell+1})-S(G_{\ell})$. We have
\begin{align*}
c_2(\ell+1)-c_2(\ell) &= -\frac{3}{8}\frac{t^2(t-1)^2}{(n-1)^2}(2\ell+1) + \frac{1}{2} \frac{t(t-1)^2}{(n-1)},\\
c_4(\ell+1) -c_4(\ell)&=  -\frac{105}{128}\frac{t^4(t-1)^4}{(n-1)^4}(4\ell^3+6\ell^2+4\ell+1) +\frac{35}{16} \frac{t^3(t-1)^4}{(n-1)^3}(3\ell^2+3\ell+1) \nonumber \\
&\hspace*{4mm}-\frac{15}{8}\frac{t^2(t-1)^4}{(n-1)^2}(2\ell+1) 
+\frac{1}{2}\frac{t(t-1)^4}{(n-1)},\\
    c_6(\ell+1)-c_6(\ell)
    &=   O\left(\frac{1}{n-1}\right).  
\end{align*}
Plugging $\ell=\ell_1\cdot \left(1+ O\left(\frac{1}{\sqrt{n-1}}\right)\right)$
into $c_4(\ell+1)-c_4(\ell)$, we have
$$ c_4(\ell+1)-c_4(\ell) = -\frac{1}{18}\frac{t(t-1)^4}{n-1} + O\lp\frac{1}{(n-1)^{3/2}}\rp.$$
Therefore,
we have
\begin{align}
    S(G_{\ell+1})-S(G_\ell) &= \frac{2(c_2(\ell+1)-c_2(\ell))}{\sqrt{n-1}}
    +  \frac{2(c_4(\ell+1)-c_4(\ell))}{(n-1)^{3/2}} 
    + \frac{2(c_6(\ell+1)-c_6(\ell))}{(n-1)^{5/2}} 
    + O\left(\frac{1}{(n-1)^{3}}\right) \nonumber \\
    &= \frac{2t(t-1)^2}{(n-1)^{5/2}}
    \left(-\frac{3}{8}t(2\ell+1) +\frac{1}{2}(n-1)
 - \frac{(t-1)^2}{18}  
    \right)
    + O\lp \frac{1}{(n-1)^3}\rp \nonumber\\
    &=-\frac{3t^2(t-1)^2}{2(n-1)^{5/2}}
    \left(\ell+\frac{1}{2} -\frac{2}{3t}(n-1)
 + \frac{2(t-1)^2}{27t}  
    \right) + O\left(\frac{1}{(n-1)^{3}}\right).    
    \label{eq:cha}
\end{align}

{\bf Case a:} $t\geq 3$ and $t$ is odd. Recall that in this case we let
$$\ell_0=\left\lfloor\frac{2n+\xi_t}{3t}\right\rfloor$$
where
$$\xi_t=\left\lfloor \frac{3t}{2}-2 - \frac{2(t-1)^2}{9}\right \rfloor. $$

For  $\ell \geq \ell_0$, we have 
\begin{align*}
\ell+\frac{1}{2} -\frac{2}{3t}(n-1)
 + \frac{2(t-1)^2}{27t} 
 &\geq \ell_0+\frac{1}{2} -\frac{2}{3t}(n-1)
 + \frac{2(t-1)^2}{27t} \\
 &\geq \frac{2n+\xi_t}{3t} -\left(1 -\frac{1}{3t}\right) +\frac{1}{2} -\frac{2(n-1)}{3t}
 + \frac{2(t-1)^2}{27t} \\
 &\geq \frac{1}{3t} \left(\xi_t+1 -  \left(\frac{3t}{2}-2 - \frac{2(t-1)^2}{9}\right)
 \right)\\
 &> 0.
 \end{align*}
Plugging it into Equation \eqref{eq:cha}, we have that for $\ell \geq \ell_0$,
$$S(G_{\ell+1})-S(G_\ell) \leq -\frac{t(t-1)^2\left(\xi_t+1-  \left(\frac{3t}{2}-2 - \frac{2(t-1)^2}{9})\right) \right)}{2(n-1)^{5/2}} + O\lp \frac{1}{(n-1)^3}\rp  <0.$$
When $\ell \leq \ell_0-1$, we have 
\begin{align*}
\ell+\frac{1}{2} -\frac{2}{3t}(n-1)
 + \frac{2(t-1)^2}{27t} 
 &\leq \ell_0 -1 +\frac{1}{2} -\frac{2}{3t}(n-1)
 + \frac{2(t-1)^2}{27t} \\
 &\leq \frac{2n+\xi_t}{3t} -1 +\frac{1}{2} -\frac{2(n-1)}{3t}
 + \frac{2(t-1)^2}{27t} \\
 &\leq \frac{1}{3t} \left(\xi_t-  \left(\frac{3t}{2}-2 - \frac{2(t-1)^2}{9}\right)
 \right)\\
 &< 0.
\end{align*}
At the last step, we observe that
$\frac{3t}{2}-2 - \frac{2(t-1)^2}{9}$ is not an integer for odd $t$. Thus, the inequality is strict.
Therefore, for $\ell \leq \ell_0-1$,
$$S(G_{\ell+1})-S(G_\ell) \geq \frac{t(t-1)^2\left(-\eta_t+  \left(\frac{3t}{4}-2 - \frac{2(t-1)^2}{9})\right) \right)}{2(n-1)^{5/2}} + O\lp \frac{1}{(n-1)^3}\rp  >0.$$

Therefore, $S(G_\ell)$ reaches the unique maximum at $\ell_0$ for sufficiently large $n$. This completes the case for odd $t$.

{\bf Case b:} $t\geq 2$ even. Let
$$\ell_0=\left\lfloor \frac{n+\eta_t}{3t/2}\right\rfloor$$
where
$$\eta_t=\left\lfloor \frac{3t}{4}-1 - \frac{(t-1)^2}{9}\right \rfloor. $$
For  $\ell \geq \ell_0$, we have 
\begin{align*}
\ell+\frac{1}{2} -\frac{2}{3t}(n-1)
 + \frac{2(t-1)^2}{27t} 
 &\geq \ell_0+\frac{1}{2} -\frac{2}{3t}(n-1)
 + \frac{2(t-1)^2}{27t} \\
 &\geq \frac{n+\eta_t}{3t/2} -\left(1 -\frac{1}{3t/2}\right) +\frac{1}{2} -\frac{2(n-1)}{3t}
 + \frac{2(t-1)^2}{27t} \\
 &\geq \frac{2}{3t} \left(\eta_t+1 -  \left(\frac{3t}{4}-1 - \frac{(t-1)^2}{9}\right)
 \right)\\
 &> 0.
 \end{align*}
Plugging it into Equation \eqref{eq:cha}, we have that for $\ell \geq \ell_0$,
$$S(G_{\ell+1})-S(G_\ell) \leq -\frac{t(t-1)^2\left(\eta_t+1-  \left(\frac{3t}{4}-1 - \frac{(t-1)^2}{9}\right) \right)}{2(n-1)^{5/2}} + O\lp \frac{1}{(n-1)^3}\rp  <0.$$

When $\ell \leq \ell_0-1$, we have 
\begin{align}
\ell+\frac{1}{2} -\frac{2}{3t}(n-1)
 + \frac{2(t-1)^2}{27t} 
 &\leq \ell_0 -1 +\frac{1}{2} -\frac{2}{3t}(n-1) 
 + \frac{2(t-1)^2}{27t} \nonumber\\
 &\leq \frac{\eta_t}{3t/2} -1 +\frac{1}{2} +\frac{2}{3t}
 + \frac{2(t-1)^2}{27t} \nonumber \\
 &\leq \frac{2}{3t} \left(\eta_t-  \left(\frac{3t}{4}-1 - \frac{(t-1)^2}{9}\right)
 \right) \nonumber\\
 &\leq 0 \label{eq:last-edge-case}.
\end{align}
 If $\frac{2}{3t}(n+\eta_t)$ is not an integer, we have 
 $$S(G_{\ell+1})-S(G_\ell) > \frac{t(t-1)^2\left(-\eta_t+  \left(\frac{3t}{4}-1 - \frac{(t-1)^2}{9}\right) \right)}{2(n-1)^{5/2}} + O\lp \frac{1}{(n-1)^3}\rp  \geq 0.$$
Therefore $\ell_0$ is the unique maximum point of $S(G_\ell)$. 

Now we assume $\frac{2}{3t}(n+\eta_t)$ is an integer.
Observe that $\frac{3t}{4}-1 - \frac{(t-1)^2}{9}$ is an integer if and only if $t$ is divisible by $4$ and $t-1$ is divisible by $3$. Therefore, the inequality \eqref{eq:last-edge-case} is strict except for the case when $t\equiv 4  \mod 12$ and $\ell=\ell_0-1$. It implies that for $t\not\equiv 4 \mod 12$,
$$S(G_{\ell+1})-S(G_\ell) \geq \frac{t(t-1)^2\left(-\eta_t+  \left(\frac{3t}{4}-1 - \frac{(t-1)^2}{9}\right) \right)}{2(n-1)^{5/2}} + O\lp \frac{1}{(n-1)^3}\rp  >0.$$
Thus for $t\not\equiv 4 \mod 12$, $S(G_\ell)$ reaches the unique maximum at $\ell_0=\left\lfloor \frac{2n+\xi_t}{3t}\right\rfloor $ for sufficiently large $n$.

Now we consider the remaining case that $t\equiv 4 \pmod{12}$ and  $\frac{2}{3t}(n+\eta_t)$ is an integer. In this case,  $S(G_{\ell})$ can only achieve the maximum at 
$\ell_0$, or $\ell_0-1$, or both. In fact, we claim both of them are maximum points.

Let $k=(t-4)/12$ and $\ell_0=\frac{2}{3t}(n+\eta_t)$. Note that by our assumption $k$ and $\ell_0$ are both integers. Rearranging the terms, we have
\begin{align}
    t      &=12k+4, \label{eq:t}\\
    \eta_t &=1+k-16k^2, \label{eta}\\
    n      &=6(3k+1)\ell_0 + 16k^2 -k -1. \label{eq:n}
\end{align}

Now we compute the spread of $G_{\ell}$ where $\ell=\ell_0$ or $\ell_0-1$. 
By Lemma \ref{l8}, $\lambda_1$ and $\lambda_n$ of $G_\ell$ satisfies the equation
\begin{align*}
    \lambda^2&=(n-1) +\sum_{k=1}^\infty \lambda^{-k}\one' A_H^k\one\\
    &=(n-1) +\sum_{k=1}^\infty \lambda^{-k}\ell t (t-1)^k\\
    &= (n-1) + \ell t \frac{(t-1)/\lambda}{1-(t-1)/\lambda}\\
    &= (n-1) + \frac{\ell t (t-1)}{\lambda - (t-1)}.
\end{align*}
Simplifying  it, we get
\begin{equation}\label{eq:P}
\lambda^3-(t-1)\lambda^2-(n-1)\lambda +(t-1)(n-1-\ell t)=0.
\end{equation}
Let us define the {\em spread of a polynomial $\phi$}, denoted by $S(\phi)$, as the difference of largest root and the smallest root. Thus, we have
$$S(G_\ell)=S(\phi_\ell), $$
where $\phi_\ell$ is defined by the left hand side of Equation
\eqref{eq:P}. Let $\lambda=x+\frac{t-1}{3}$. The cubic equation \eqref{eq:P} can be written as
\begin{equation} \label{eq:Preduced}
    x^3 -\frac{1}{3}(n+t^2-2t-2)x + \frac{1}{27}(-27  l t^{2} - 2 t^{3} + 27 l t + 18  n t + 6  t^{2} - 18 n - 24  t + 20)=0.
\end{equation}
Now plugging $\ell=\ell_0$, $t$ as in Equation \eqref{eq:t}, and
$n$ as in Equation \eqref{eq:n}, into Equation \eqref{eq:Preduced}, we get
\begin{equation} \label{eq:Pl0}
    x^3 - (6(3k+1)\ell_0 +64k^2 +23k+1) x -(72k^2+42k+6)=0.
\end{equation}
Similarly, plugging $\ell=\ell_0-1$, $t$ as in Equation \eqref{eq:t}, and
$n$ as in Equation \eqref{eq:n}, into Equation \eqref{eq:Preduced}, we get
\begin{equation} \label{eq:Pl0-1}
    x^3 - (6(3k+1)\ell_0 +64k^2 +23k+1) x+(72k^2+42k+6)=0.
\end{equation}
Let the $\phi_1$ (or $\phi_2$) denote the cubic polynomial in the left hand of Equation \eqref{eq:Pl0} (or Equation
\eqref{eq:Pl0-1} respectively). 
Observe that $\phi_2(x)=-\phi_1(-x)$. If
$\phi_1$ has three real roots $x_1\leq x_2\leq x_3$, then
$\phi_2$ has three real roots $-x_3\leq -x_2\leq -x_1$.
Thus $$S(\phi_1)=x_3-x_1=(-x_1)-(-x_3)=S(\phi_2).$$
It then follows that 
$$S(G_{\ell_0})= S(\phi_{\ell_0}) =S(\phi_1) =S(\phi_2)=S(\phi_{\ell_0-1})=S(G_{\ell_0-1}).$$
Therefore both $G_{\ell_0}$ and $G_{\ell_0-1}$ are extremal graphs for this special case. This completes the proof of Theorem \ref{thm:main}.
\end{proof}

\section*{Acknowledgement}
We thank the anonymous referee for a number of useful comments which improved the presentation of the manuscript.

\end{document}